\documentclass[eqno,a4paper,11pt]{article}
\title{Local solvability for a quasilinear wave equation with the far field degeneracy: 1D case}
 \author{Yuusuke Sugiyama\footnote{e-mail:sugiyama.y@e.usp.ac.jp\ The University of shiga prefecture}}
\date{}
\usepackage{amssymb,amsmath,amsthm,mathrsfs,delarray,enumerate}
\usepackage{graphics}
\usepackage{braket}

\theoremstyle{definition} 
\newtheorem{Def}{Deffinition}[section]
\newtheorem{Prop}[Def]{Proposition}
\newtheorem{lemma}[Def]{Lemma}
\newtheorem{theorem}[Def]{Theorem}
\newtheorem{remark}[Def]{Remark}

\newcommand{\N}{\mathbb{N}} 
\newcommand{\R}{\mathbb{R}} 

\newcommand{\al}{\alpha} 
\newcommand{\bt}{\beta} 
\newcommand{\bR}{\bar{R}} 
\newcommand{\bS}{\bar{S}} 

\makeatletter

\@addtoreset{equation}{section}
\newcommand{\bV}{\bar{V}} 
\newcommand{\bW}{\bar{W}} 

\begin{document}
\maketitle %
\begin{abstract}
We study the Cauchy problem for the  quasilinear wave equation $ \partial^2 _t u = u^{2a} \partial^2_x u  + F(u) u_x $ with $a \geq 0$ and show a result for the local in time existence under new conditions.
In the previous results, it is assumed that $u(0,x) \geq c_0>0$ for some constant $c_0$ to prove the existence and  the uniqueness.
This assumption ensures  that the equation does not degenerate.
In this paper, we allow the equation to degenerate at spacial infinity.  
Namely we consider the local well-posedness under the assumption that  $u(0,x)>0$ and $u(0,x) \rightarrow 0$ as $|x| \rightarrow \infty$. 
Furthermore, to prove the local well-posedness, we find that the so-called Levi condition appears.
Our proof is based on the method of characteristic and the contraction mapping principle via weighted $L^\infty$ estimates.
\end{abstract}
%

\section{Introduction}
In this paper, we consider the following Cauchy problem of the model quasilinear wave equation in $\R$: 
\begin{eqnarray} 
\left\{  \begin{array}{ll} \label{req}
   \partial^2 _t u = \partial_x ( u^{2a} \partial_x u)  + F(u)u_x ,  \ \ (t,x) \in (0,T] \times \R, \\   
  u(0,x) = u_0 (x),\ \ x \in \R,                \\            
    \partial_t u(0,x) = u_1 (x), \ \ x \in \R,  
\end{array} \right.  
\end{eqnarray} 
where $F$ is a given function and  $a \geq 0$ .
The purpose of this paper is  to show the local existence and the uniqueness under new conditions.
The existence of solutions to the more general quasi-linear wave equations has been widely known since the 1970s.
Kato \cite{tk} and  Hughes, Kato and Marsden \cite{HKM} have shown an abstract theorem about the well-posedness 
of the system of  general quasi-linear wave equations in $L^2$ Sobolev space.
In $1$ dimensional case, the well-posedness in $C^1 _b$ class for first order  hyperbolic equations has been studied  by Douglis \cite{D} and Hartman and Winter \cite{HW2} (see also Majda \cite{m} and Courant and Lax \cite{CL}),
where $C^1 _b $ is a set of continuous and bounded functions whose derivatives are also bounded.
In order to apply these results to the existence problem of \eqref{req}, the following assumption is required:
\begin{eqnarray} \label{non-deg}
 u_0 (x) \geq   c_0 >0
\end{eqnarray}
for a constant $c_0$. This condition ensures that the equation in \eqref{req} is the strictly hyperbolic type near $t=0$.  This paper relaxes this condition.  
We show the  the local existence and the uniqueness of solutions of  \eqref{req} under the assumption that the equation degenerates at spacial infinity.
Namely we weaken \eqref{non-deg} by $u(0,x)>0$ and allow  that  $u(0,x) $ can decay to $0$ as $|x| \rightarrow \infty$ (more precise assumptions are given later).
To the best of my knowledge, the well-possessedness  has never been studied under these types of assumptions.

\subsection{Known results}
Let us review some results on the solvability for degenerate wave equations (weakly hyperbolic equations).
The existence, nonexistence and regularity of solutions to the following type of  linear weakly hyperbolic equations have been studied by many authors (e.g. Oleinik \cite{ol}, Colombini and  Spagnolo  \cite{CS}, Ivrii and Petkov \cite{IV} and Taniguchi and  Tozaki \cite{TT}), 
\begin{eqnarray} \label{wl}
\partial^2 _t u - \sum_{i,j=1} ^n  a_{i,j}(t,x)   u_{x_i x_j} + \sum_{j=1} ^n b_j (t,x)  u_{x_j}=0,
\end{eqnarray}
where $a_{i,j}$ and $b_j$ are smooth functions and $\sum_{i,j=1} ^n  a_{i,j}(t,x) \xi_i \xi_j \geq 0$ is assumed for $(\xi_1 , \ldots , \xi_n) \in \R^n$. We note that  $\sum_{i,j=1} ^n a_{i,j}(t,x) \xi_i \xi_j =0$ corresponds to the degeneracy.
In Oleinik \cite{ol}, \eqref{wl} have been solved under the so-called Levi condition: 
\begin{eqnarray} \label{levi}
  C_1   \left(  \sum_{j=1}^n b_j \xi_j \right)^2 \leq C_2 \left( \sum_{i,j=1} ^n a_{i,j} \xi_i \xi_j + \partial_t a_{i,j} \xi_i \xi_j  \right)  .
\end{eqnarray}
Even if  the Levi condition is assumed, we can only obtain  the following energy estimate with the regularity loss for weakly hyperbolic equations:
\begin{eqnarray} \label{re-loss}
\| u \|_{H^{s}} + \| u_t \|_{H^{s-1}} \leq C(\| u_0 \|_{H^{s+r_1}} + \| u_1 \|_{H^{s-1 +r_2}}),
\end{eqnarray}
where $s$ is an arbitrary real number and $r_1$ and $r_2$ are  non-negative numbers. It is known that this estimate is optimal in the sense of the regularity by observing some
explicit solution to some special linear weakly hyperbolic equations. 
Ivrii and Petkov in \cite{IV} have treated the the following model of the $1D$ weakly hyperbolic equation:
\begin{eqnarray*}
u_{tt} - t^{2l} u_{xx} +t^k u_x =0 .
\end{eqnarray*}
They have shown that the Levi condition ($k \geq l-1$) is  necessary  for the Cauchy problem of this equations to be $C^\infty$ well-posed.
Colombini and  Spagnolo in \cite{CS} have given an example of a $C^\infty$ function $a(t) \geq 0$ such that
$$
u_{tt} - a(t)   u_{xx}=0
$$
is not well-posed in $C^\infty$. Roughly speaking,  highly oscillatory behaviors of $a(t)$ near the point that $a(t)=0$ causes the ill-posedness.  In \cite{qh}, Han has derived an energy inequality with a regularity loss for the linear weakly hyperbolic equation:
\begin{eqnarray*}
\partial^2 _t u - a(t,x)  u_{xx} =0,
\end{eqnarray*}
where $a(t,x) = t^m + a_1(x) t^{m-1} + a_2 (x) t^{m-2} + \cdots +a_{m-1} (x) t + a_m (x) $.

Manfrin in \cite{rm2} have established the local existence and the uniqueness for following  $1D$ degenerate quasilinear wave equations
with $u_0, u_1  \in C^\infty _0 (\R^n)$:
\begin{eqnarray*}
u_{tt} = a(u) \Delta u,
\end{eqnarray*}
where $a(u)$ is a analytic function satisfying $a(0)=0$. This result can be extended to more general degenerate wave equations (see also Manfrin \cite{rm0, rm1}).
In Dreher's paper \cite{dre},  he also has shown the local solvability for  $\partial^2 _t u =\partial_x (|\partial_x u|^{p-2} \partial_x)$ with $p >5$ and under the initial condition that
$u_0, u_1  \in C^k _0 (\R^n)$ for a large natural number $k$.
 Since their proof is based on the Nash-Moser implicit function theorem and the argument in the Oleinik's paper \cite{ol},  the compactness of the support of initial data is essentially 
used. Hence it does not seem difficult to extend Manfrin's method to the case that initial data are not compactly supported.

In \cite{HW}, Hu and Wang have shown the local existence and uniqueness of solutions to the following  variational wave equation:
\begin{eqnarray} \label{v-hen}
\partial^2 _t u = c(u,x) \partial_x ( c(u,x) \partial_x u)
\end{eqnarray}
with initial data  and the function $c(u,x)$ satisfying
\begin{align*}
c(u(0,x),x) =0,\\
\partial_t u(0,x) \geq c_0 >0, \\
c_u (u,x) \geq c_1 >0
\end{align*}
for some constants $c_0$ and $c_1$. The choice of initial data implies that the equation degenerates at $t=0$ and  that $c(u,x)$ becomes positive uniformly and immediately after $t=0$.
The method in \cite{HW} is inspired by Zhang and Zheng's paper \cite{tyzy} which studies  the existence of solutions to Euler type equation in gas dynamics. In \cite{tyzy} and \cite{HW}, they use method of characteristic for a new dependent variable
and the fixed point theorem in a special metric space.

\subsection{Assumptions and main theorem}

Before stating main theorem of this paper, we introduce assumptions on initial data and the function $F$.
We set $\gamma =\gamma (a, \al)$ as below:
\begin{eqnarray*}
\gamma = \left\{
\begin{array}{ll}
0, & a \geq 1, \\
(1-a)\al, &  \mbox{otherwise}.
\end{array}
\right.
\end{eqnarray*}
For initial data $u_0 \in C^2  (\R)$ and $u_1 \in C^1 _b (\R)$, we assume that 
\begin{align}  
 c_1 \braket{x}^{-\alpha}  \leq  u_0 (x) \leq  c_2, \label{ini-con} \\
 |u_{1} (x) \pm u_{0} ^a u'_{0} (x)| \leq c_3  \braket{x}^{-\beta}, \label{ini-con2}\\
  \left|\frac{d}{dx} (u_{1} (x) \pm u_{0} ^a u'_{0} (x)) \right| \leq c_4\braket{x}^{-\gamma} \label{ini-con4}
\end{align}
with conditions on $\alpha \geq 0$ and $\beta \geq 0$  that
\begin{align}  
\al \leq \bt, \label{ab1}\\
a\al  \leq \bt, \label{ab2}
\end{align}
where $u' _0=d u_0 /dx$ and $\braket{x}$ is defined by $ \braket{x}=(1+x^2)^{1/2}$ and  $c_1, c_2, c_3, c_4$ are positive constant.
The assumption \eqref{ini-con} indicates that the equation in \eqref{req}  can degenerate at spatial infinity.
For the function $F \in C([0,\infty)) \cap C^1 ((0,\infty))$, we assume that
\begin{eqnarray} 
 |F(\theta)| \leq C_K \theta^a, \label{f-con}\\
 |F'(\theta)| \leq C_K \theta^{a-1} \label{f-con2}
\end{eqnarray}
for $\theta \in (0,K]$ and $C_K$ is a positive constant. Typical example of $F$  is $F(\theta)= \theta^b$ with $b \geq a$. This condition 
appears in the study of weakly hyperbolic equations called a sufficient Levi condition  (e.g. Manfrin's paper \cite{rm2}) .
Main theorem of this paper is as follows.
\begin{theorem} \label{main}
Let $u_0 \in C^2  (\R) $ and $u_1 \in C^1 _b (\R)$. Suppose that the conditions \eqref{ini-con}-\eqref{f-con2} hold.
Then there exists a number $T>0$ depending on  the constants in \eqref{ini-con}-\eqref{f-con2} such that the Cauchy problem \eqref{req} has a unique local solution $u \in C^2 ([0,T]\times \R)  $ satisfying  that for all $(t,x) \in [0,T ] \times \R$
\begin{eqnarray}
 C_1 \braket{x}^{-\alpha} \leq u(t,x) \leq C_2,  \label{dc1} \\
|(u_t  \pm u^a u_x )(t,x)| \leq C_3  \braket{x}^{-\beta}, \label{dc2} \\
|\partial_t \left((u_t  \pm u^a u_x) (t,x)\right)| +|\partial_t \left((u_t  \pm u^a u_x) (t,x)\right)| \leq  C_4\braket{x}^{-\gamma},\label{dc3}
\end{eqnarray}
where $C_1, C_2, C_3, C_4$ are positive constants. 
\end{theorem}
Theorem \ref{main} asserts the local existence and uniqueness of solutions of \eqref{req} under 
the Levi type condition without the regularity loss.
Our proof is based on the method of characteristic and the contraction mapping principle via weighted $L^\infty$ estimates.
In contrast to previous results on the existence for strictly hyperbolic equations, $1/u$ is not bounded . 
To avoid this crux, we use the spatial decay of $u_t \pm u^a u_x$. In particularly, this property helps to show the boundedness of the derivative of characteristic curves
$x_{\pm}(t)$ with initial position (see Lemma \ref{x-d-lm}). We also remark that our approach is applicable to various type of 1D quasilinear wave equations (e.g. the equation as \eqref{v-hen} under suitable condition on $c(u,x)$).

\begin{remark}
Suppose that initial data is $(u_0(x), u_1 (x))=(\braket{x}^{-\alpha_1} , \braket{x}^{-\alpha_2})$ with $\al_1, \al_2 \geq 0$.
If $\al_2 \geq \al_1$ and $a\al_1 \leq \al_2$ are satisfied, then  all assumptions  \eqref{ini-con}-\eqref{ab2} on initial data are satisfied. We also remark that the condition \eqref{ab2} is not necessary in the case that $a\leq 1$, since \eqref{ab1} implies \eqref{ab2}. While if $a \geq 1$,  \eqref{ab1} is not necessary.
\end{remark}

\subsection{Notation and plan of the paper}

For a domain $\Omega \subset \R^n$, we define $C^m _b(\Omega)$ with $m \in \N$ as follows
\begin{eqnarray*}
C^m _b (\Omega) = \{ f \in C^m (\Omega) \ | \   \sum_{|\alpha| \leq m}\sup_{x \in \Omega} | \partial_x ^\alpha f(x)| <\infty \}.
\end{eqnarray*}
We write $C _b (\Omega) =C^0 _b  (\Omega)$ and denote the Lebesgue space for $1\le p\le \infty$ on $\R^n$ by $L^{p} $ with the norm $\| \cdot \|_{L^p }$.
For a Banach space $X$, $1\le p\le \infty$ and $T>0$, we denote the set of all $X$-valued $L^p$ functions with $t \in [0,T]$
by $L^p([0,T]; X)$. For convenience, we denote $L^p([0,T]; X)$ by $L^p _T X$. The norm of $L^p _T X$ is denoted by $\|f \|_{L^p _T X}$. 
Various constants are simply denoted by $C$ or $C_j$ for $j \in \N$. 
We denote that $\braket{x}=(1+x^2)^{1/2}$.

The remainder of the present paper is organized as follows. 
In Section 2, we review several  formulas for the unknown valuable $R=u_t + u^a u_x$ and $S= u_t - u^a u_x$, which are called Riemann invariant in the study of the 1D hyperbolic conservation law, and give some estimates for characteristic curves. 
In Section 3, we show Theorem \ref{main} by using the method of characteristic, weighted $L^\infty$ estimates and the contraction mapping principle.
Concluding remarks are given in Section 5.

\section{Preliminaries}

\subsection{Basic formulation for unknown variables $R$ and $S$}
We set $R(t,x)$ and $S(t,x)$ as follows
\begin{eqnarray}\label{ri}\left\{
\begin{array}{ll}
R =\partial_t u +u^a \partial_x u, \\
S =\partial_t u - u^a \partial_x u.
\end{array}\right.
\end{eqnarray}
By \eqref{req}, $R$ and $S$ are  solutions to the system of the following first order equations:
\begin{eqnarray}\label{fs}
 \left\{  
\begin{array}{ll}
\partial_t R -u^a\partial_x R=N_1(u,R,S) + L(u,R,S) , \\
\partial_x u = \dfrac{1}{2u^a}(R - S),\\
\partial_t S +u^a\partial_x S =N_2(u,R,S) + L(u,R,S),
\end{array}
\right.
\end{eqnarray} 
where we set
$$
L(u,R,S)=\dfrac{F(u)(R-S)}{2u^a},
$$
$$
N_1(u,R,S)= \dfrac{a}{2u}(R^2 -RS)
$$
and
$$
N_2(u,R,S)= \dfrac{a}{2u}(S^2 -RS).
$$
Let $x_{\pm} (t)$ be characteristic curves on the first and third equations of \eqref{fs}
respectively. That is,  $x_{+} (t)$ and $x_{-} (t)$ are solutions to the following differential equations respectively:
 \begin{eqnarray}\label{cc}
\dfrac{d}{dt} x_{\pm} (t)=\pm u^a(t,x_{\pm} (t)).
\end{eqnarray}
When we emphasize the characteristic curves go through $(s,y)$, we denote $x_{\pm} (t)$ by  $x_{\pm} (t; s,y)$. 
That is, $x_{\pm} (t; s,y)$ satisfies that
\begin{eqnarray} \label{int-x}
x_{\pm} (t; s,y) = y \pm \int_s ^t  u^a (\tau, x_{\pm} (\tau; s,y)) d\tau. 
\end{eqnarray}
On the characteristic curves, $R$ and $S$ satisfy that
\begin{eqnarray}
 \left\{  
\begin{array}{ll}
\dfrac{d}{dt} R (t,x_- (t)) =N_1(u,R,S)(t,x_- (t)) + L(u,R,S)(t,x_- (t)) , \\
\dfrac{d}{dt} S (t,x_+ (t))  =N_2(u,R,S) (t,x_+ (t)) + L(u,R,S)(t,x_+ (t)).\label{fs-ch}
\end{array}
\right.
\end{eqnarray} 
\subsection{Some estimates of characteristic curves}
We prepare some estimates for characteristic curves for $u \in C^1  ([0,T]\times \R) $ satisfying for $\al \geq 0$
\begin{align} \label{u-pro}
 \braket{x}^{-\al} A_0 \leq  u(t,x) \leq A_1 ,
\end{align}
where $A_0$ and $A_1$ are positive constants.
In addition, we assume that
\begin{align} \label{u-pro2}
0 \leq  u^a |u_x (t,x)| \leq A_2 \braket{x}^{-\al}
\end{align}
for a constant $B_1$. 
The boundedness of $u$   and \eqref{int-x} implies the following estimate with $s,t \in [0,T]$:
 \begin{align} \label{x-es}
x -  A_1 ^a|t-s| \leq x_{\pm} (s;t,x)\leq  x + A_1 ^a|t-s|.
\end{align}
Next we show a lemma ensures a uniform Lipschitz continuity of $x_{\pm} (t;s,y )$.
This lemma helps to show  that a sequence of characteristic curves satisfies an assumption of  the Arzel\'a-Ascoli theorem.
\begin{lemma} \label{lip-x-lm}
Let $u \in C^1  ([0,T]\times \R) $. Suppose that \eqref{u-pro}  and  \eqref{u-pro2} hold.  Then the characteristic curves fulfill that for $x_1, x_2 \in \R$ and $t_1 , t_2,  t_3, t_4 \in [0,T]$ 
\begin{eqnarray} \label{lip-x}
|x_{\pm} (t_3; t_1,x_1) -  x_{\pm} (t_4; t_2,x_2 )| \leq 3 (1+A_1 ^{a})( |x_1 - x_2| +  |t_1 - t_2| + |t_3 - t_4|),
\end{eqnarray}
if $T>0$ is sufficiently small.
\end{lemma}
\begin{proof}
First we show the case that $t_3 = t_4 =t \in [0,T]$ and $t \geq t_1, t_2$. 
From \eqref{int-x}, we can easily  compute that
\begin{align}
|x_{\pm} (t; t_1,x_1) -  x_{\pm} (t; t_2,x_2 )| \leq  & |x_1 - x_2| \notag \\
 + &  \left|\int_{t_1}  ^{t} u^a (\tau, x_{\pm} (\tau; t_1,x_1)) d\tau -   \int_{t_2}  ^{t} u^a (\tau, x_{\pm} (\tau; t_1,x_2)) d\tau   \right| \notag \\
\leq  & |x_1 - x_2| + \left|\int  ^{t_1} _{t_2} u^a (\tau, x_{\pm} (\tau; t_2,x_2)) d\tau \right| \notag  \\
+ &    \int_0 ^{t}  \left| u^a (\tau, x_{\pm} (\tau; t_1,x_1))   -   u^a (\tau, x_{\pm} (\tau; t_2,x_2))  \right| d\tau. \label{x-lip-es}
\end{align}
From \eqref{u-pro} and \eqref{u-pro2}, we have that
$u^{a-1} u_x $ is bounded,
from which, we have for the third term of the right hand side in \eqref{x-lip-es} that
\begin{align*} 
\left| u^a (\tau, x_{\pm} (\tau; t_1,x_1))   -   u^a (\tau, x_{\pm} (\tau; t_2,x_2)) \right| \leq & \left| \int_{x_{\pm} (\tau; t_2,x_2) } ^{x_{\pm} (\tau; t_1,x_1) } u^{a-1} u_x (\tau,y) dy  \right| \\
\leq & C  \left| x_{\pm} (\tau; t_1,x_1)   -    x_{\pm} (\tau; t_2,x_2)  \right|.
\end{align*}
Hence we have  
\begin{align*}
|x_{\pm} (t; t_1,x_1) -  x_{\pm} (t; t_2,x_2 )| \leq  & |x_1 - x_2| + A_1 ^{\al} |t_1 - t_2 | \\
+ &    C\int_0  ^{t}  \left| x_{\pm} (\tau; t_1,x_1)   -    x_{\pm} (\tau; t_2,x_2)  \right| d\tau. \\
\leq  & (1+A_1 ^a)  (|x_1 - x_2| +  |t_1 - t_2 |) \\
 + &   C \int_0  ^{t}  \left| x_{\pm} (\tau; t_1,x_1)   -    x_{\pm} (\tau; t_2,x_2)  \right| d\tau. 
\end{align*}
Thus we have \eqref{lip-x} from the Grownwall inequality
\begin{eqnarray*}
|x_{\pm} (t; t_1,x_1) -  x_{\pm} (t; t_2,x_2 )| \leq  (1+A_1 ^a)  (|x_1 - x_2| +  |t_1 - t_2 | )e^{Ct} .
\end{eqnarray*}
Hence if $T$ is small, then we have that with $t_1 \geq t_2$ 
\begin{eqnarray} \label{x-tt}
|x_{\pm} (t; t_1,x_1) -  x_{\pm} (t; t_2,x_2 )| \leq 2 (1+A_1 ^a)  (|x_1 - x_2| +  |t_1 - t_2 | ).
\end{eqnarray}
In the same way as above, we can show \eqref{x-tt} with the case that $t <t_1$ or $t <t_2$. We omit the proof of this case.
Next we show \eqref{lip-x}. The left hand side of \eqref{lip-x} is written by
\begin{align*}
|x_{\pm} (t_3; t_1,x_1) -  x_{\pm} (t_4; t_2,x_2 )| \leq & |x_{\pm} (t_3; t_1,x_1) -  x_{\pm} (t_3; t_2,x_2 )| \\
& + |x_{\pm} (t_3; t_1,x_1) -  x_{\pm} (t_4; t_1,x_x )| 
\end{align*}
From \eqref{x-tt}, the first term of the right hand side is estimated by $ 2 (1+A_1 ^a)  (|x_1 - x_2| +  |t_1 - t_2 | )$.
From \eqref{int-x} and \eqref{u-pro}, the second term is estimated by $A_1 ^a |t_3 - t_4|$. Therefore, we have the desired inequality.
\end{proof}
Following lemma is used to show the boundedness of the derivatives of $R$ and $S$.
\begin{lemma} \label{x-d-lm}
Let $u \in C^1  ([0,T]\times \R) $. Suppose that \eqref{u-pro}  and  \eqref{u-pro2} hold. Then the characteristic curves $x_{\pm} (t; s,x)$ are differentiable with $x$ and $\partial_x x_{\pm} (t; s,x)$ satisfies that with $(t,x) \in [0,T] \times \R$ and $s \in [0,T]$ for small $T>0$
\begin{eqnarray}\label{lve-eq}
 \left\{  
\begin{array}{ll}
\dfrac{d}{ds}\partial_x x_{\pm} (s;t,x) =\pm a u^{a-1} u_x (t, x_{\pm} (s;t,x)) \partial_x x_{\pm} (s;t,x), \\
\partial_x x_{\pm} (t;t,x)=1
\end{array}
\right.
\end{eqnarray}
and
\begin{eqnarray} \label{x-d-es}
|\partial_x x_{\pm} (s;t,x)| \leq e^{C|t-s|},
\end{eqnarray}
where the positive constant $C$ is depending on $A_0, A_1$ and $A_2$. 
\end{lemma}
\begin{proof}
The differentiability of $ x_{\pm} (s;t,x)$ and \eqref{lve-eq} are  well-known as a basic fact(e.g. textbook of Sideris \cite{sid}).
We estimate $\partial_x x_{-} (s;t,x)$. We only show \eqref{x-d-es} with the case that $t\geq s$. From \eqref{u-pro} and the boundedness of $\braket{x}^\al u^a u_x$, we obtain that
\begin{align*}
|\partial_x x_{\pm} (s;t,x) | \leq & 1 + a  \int_s  ^t |u^{a-1} u_x | |\partial_x x_{\pm} (\tau ;t,x) | d\tau  \\
\leq & 1+ C  \int_s ^t   |\partial_x x_{\pm} (\tau ;t,x) | d\tau .
\end{align*}
Hence, from the Gronwall inequality, we obtain \eqref{x-d-es} for small $T$.
\end{proof}

\section{Proof of the main theorem}
As in Introduction, we set
\begin{eqnarray*}
\gamma = \left\{
\begin{array}{ll}
0, & a \geq 1, \\
(1-a)\al, &  \mbox{otherwise}.
\end{array}
\right.
\end{eqnarray*}
We treat functions satisfying the following conditions for $\al, \bt \geq 0$  with $\al \leq \bt$ such that
\begin{align} 
 A_0 \braket{x}^{-\al} \leq  f(t,x) \leq  A_1  \label{c1}
\end{align}
and
\begin{align} 
f^a (t,x) |f_x (t,x)| \leq  A_2  \braket{x}^{-\bt}, \label{c1-2} \\
|f_t (t,x)| \leq  A_3  \braket{x}^{-\bt} \label{c1-3}
\end{align}
or
\begin{align}
 |f(t,x)|  \leq A_3 \braket{x}^{-\bt}  \label{c2},
\end{align}
and
\begin{align}
|f_x (t,x)|   \leq A_4 \braket{x}^{-\gamma}, \label{c3} \\
 |f_t(t,x)|  \leq A_5 \braket{x}^{-\gamma}, \label{c4}
\end{align}
where $A_j$ are positive constants with $j=1,\cdots,5$.
We define sets of $C^1$ functions $X_\al$, $Y_{\bt, 1}$, $Y_{\bt,2}$  as follows:
\begin{eqnarray*}
X_\al=\{ f \in C^1 \cap C_b \ | \ f(0,x)=u_0(x) \ \mbox{and \ \eqref{c1}, \eqref{c1-2} and  \eqref{c1-3} hold.}  \}, \\
Y_{\bt, 1}=\{ f \in C^1 _b \ |  \ f(0,x)=R_0(x)  \ \mbox{and \ \eqref{c2}, \eqref{c3} and  \eqref{c4} hold.}  \},\\
Y_{\bt, 2}=\{ f \in C^1 _b \ |   \ f(0,x)=S_0(x)  \ \mbox{and \ \eqref{c2}, \eqref{c3} and  \eqref{c4} hold.}  \},
\end{eqnarray*}
where given functions $(u_0, R_0 ,S_0)$ belongs to $C^1 \cap C_b \times C^1 _b \times C^1 _b$.
For given functions $(v, \bar{R}, \bar{S}) \in X_\al  \times Y_{\beta, 1}  \times Y_{\beta, 2} $, we consider the first order linear hyperbolic equation:
\begin{eqnarray}
 \left\{  
\begin{array}{ll}
 R_t -v^a R_x=N_1(v, \bR,\bS) + L(v, \bR,\bS), \\
 S_t +  v^a S_x =N_2(v, \bR, \bS) + L(v, \bR,\bS) \label{lfs}
\end{array}
\right.
\end{eqnarray} 
with initial condition $(R(0,x), S(0,x))=(R_0 , S_0) \in C^1 _b \times C^1 _b$. We set
\begin{align} \label{def-u0}
u = u_0 (x) + \int_0 ^t \frac{R+S}{2} (s,x) ds.
\end{align}
We find  that \eqref{lfs} with $C^1 $ initial data has unique and time-global solutions such that
$ R,S \in C^1 ([0,T]\times \R) \cap C_b ([0,T]\times \R)$ with arbitrary fixed $T>0$ from the method of characteristic. From \eqref{def-u0}, it holds that $u \in C^1 \cap C_b$.
Namely we can define the map 
$$
\Phi  :  X_\al    \times Y_{\beta, 1}   \times Y_{\beta, 2}    \rightarrow  C^1   \times C^1  \times C^1 
$$ such that
$\Phi (v,\bar{R}, \bar{S}) =(u,R,S) $.
We take four positive numbers  $A_0, A_1, A_3, A_4$  satisfying that
\begin{align} 
2 A_0 \braket{x}^{-\al} \leq  u_0 (x) \leq \frac{A_1}{2}, \label{0con2} \\ 
\|\braket{x}^\beta R_0 \|_{L^\infty} + \| \braket{x}^\beta S_0 \|_{L^\infty} \leq \frac{A_3}{4},\label{0con3} \\ 
\| \braket{x}^\gamma R'_0\|_{L^\infty} + \| \braket{x}^\gamma S'_0\|_{L^\infty} \leq \frac{A_4}{8}. \label{0con4} 
\end{align} 
The constants $A_2$ and $A_5$ in \eqref{c4} will be taken later.
Moreover, we  assume  that
\begin{align}\label{0con5}
\|\braket{x}^{\bt} u_0 ^a u' _0 \|_{L^\infty} \leq B_1 
\end{align}
for a positive constant $B_1$.
In the following, we show that $(u,R,S) \in  X_\al  \times Y_{\beta, 1} \times Y_{\beta, 2} $  and $\Phi$ is a contraction mapping
in the topology of $L^\infty$  for sufficient small $T$. $X_\al$ and $Y_{\beta, j}$ with $j=1,2$ are not closed set of  $L^\infty$ space. Nevertheless it is possible to show that 
the fixed point belongs to $X_\al  \times Y_{\beta, 1} \times Y_{\beta, 2}$.  Furthermore we will show that  the regularity is improved as $  u \in C^2 $.
 First we show the following proposition.
 \begin{Prop} \label{phi-def}
Let $(u_0, R_0, S_0) \in \ C^1 \times  C^1 \times C^1$ satisfying \eqref{0con2}-\eqref{0con4} and \eqref{0con5}. Suppose that $v,\bar{R}, \bar{S} \in X_\al  \times Y_{\beta, 1}  \times Y_{\beta, 2} $. Then $\Phi(v,\bar{R}, \bar{S} )=(u,R,S) \in X_\al \times Y_{\beta, 1}   \times Y_{\beta, 2} $ 
for sufficiently small $T>0$. 
\end{Prop}
\begin{proof}
 From the method of characteristic, we can see that the solution of  \eqref{lfs} can be written by
\begin{eqnarray}
\left\{  
\begin{array}{ll} 
R(t,x) = R(0,x_- (0)) + \int_0 ^t N_1(v, \bR,\bS)(s,x_- (s))+ L(v, \bR,\bS)(s,x_- (s)) ds,\\
S(t,x) = S(0,x_+ (0)) + \int_0 ^t N_2(v, \bR, \bS)(s,x_+ (s)) + L(v, \bR,\bS)(s,x_+ (s))ds,\label{lfs2}
\end{array}
 \right.
\end{eqnarray} 
where the characteristic curves for the linear equation \eqref{lfs} are defined as follows:
\begin{eqnarray*}
\dfrac{d}{dt} x_{\pm} (t) =\pm v^a (t,x_{\pm} (t))
\end{eqnarray*}
with initial data $x_{\pm}(t)=x$. From this expression, we have that $(u,R,S) \in  C^1 \times   C^1 \times   C^1$.
 Now we estimate $\|\braket{x}^\bt R\|_{L^\infty _T L^\infty}$.  From \eqref{x-es}, if $T$ is small, we have
\begin{eqnarray*}
\frac{\braket{x}^\beta}{2} \leq \braket{x_{-} (s;t,x)}^\beta \leq 2 \braket{x}^\beta .
\end{eqnarray*}
Using these inequalities, from \eqref{lfs2}, we have that if $T$ is small
\begin{align*}
|\braket{x}^\beta R(t,x)| \leq &  2 \|\braket{x}^\beta R (0,\cdot )\|_{L^\infty} \\
& +\int_0 ^t  \dfrac{\braket{x_{-} (s)}^\beta}{v} |\bR^2 -\bR\bS|  + \braket{x_{-} (s)}^\beta \frac{F(v)}{v^a}|\bR-\bS|ds \\
\leq &   \frac{A_3}{2} +\int_0 ^t  \dfrac{\braket{x_{-} (s)}^{\al + \beta}}{\braket{x_{-} (s)}^\al v} |\bR^2 -\bR\bS| +  C\braket{x_{-} (s)}^\beta |\bR-\bS|ds \\
\leq   &  \frac{A_3}{2} + CT \left( \| \braket{x}^\bt  \bR \|_{L^\infty _T L^\infty} \| \braket{x}^\bt  \bS (s)\|_{L^\infty _T L^\infty} + \|\braket{x}^\bt   \bS \|^2_{L^\infty _T L^\infty} \right) \\
 &+ C T (\| \braket{x}^\bt  \bR \|_{L^\infty _T L^\infty} + \|\braket{x}^\bt   \bS \|_{L^\infty _T L^\infty})  \\
& \leq     \frac{A_3}{2} + C T,
\end{align*}
where $C$ is a positive constant depending on $A_0, A_1, A_3$    and the assumptions that $\beta \geq \al$ and \eqref{f-con} are used.
Hence we obtain that for sufficiently small $T$
 \begin{eqnarray} \label{R-es}
\|\braket{x}^\bt R\|_{L^\infty _T L^\infty} \leq   A_3 .
\end{eqnarray}
Similarly for $S$, we have that  
 \begin{eqnarray} \label{S-es}
\|\braket{x}^\bt S\|_{L^\infty _T L^\infty}\leq   A_3.
\end{eqnarray}
Next we estimate $\| \braket{x}^\gamma  R_x \|_{L^\infty _T L^\infty}$ and $\| \braket{x}^\gamma  S_x \|_{L^\infty _T L^\infty}$. 
Differentiating the both side of the equations  \eqref{lfs} with  $x$, we can obtain integral equations for  $R_x$ and $S_x$ as follows:
\begin{align}
V (t,x) & = V_0 (x_- (0;t,x)) \partial_x x_{-} (0;t,x)  \notag \\
 + &\int_0 ^t \partial_x x_{-} (s;t,x) \left(N_{1u} v_x + N_{1R} \bV + N_{1S}  \bW\right) (t,x_- (s;t,x)) ds \notag \\
 +&  \int_0 ^t \partial_x x_{-} (s;t,x) \left(L_u v_x  + L_R \bV + L_S \bW \right) (t,x_- (s;t,x)) ds \label{V-eql}
\end{align}
and
\begin{align}
W (t,x) & = W_0 (x_- (0;t,x)) \partial_x x_{+} (0;t,x)  \notag \\
 + &\int_0 ^t \partial_x x_{+} (s;t,x) \left(N_{2u} v_x + N_{2R} W + N_{2S}  V\right) (t,x_+ (s;t,x)) ds \notag \\
 +&  \int_0 ^t \partial_x x_{+} (s;t,x) \left(L_u v_x  + L_R \bV + L_S \bW \right) (t,x_+ (s;t,x)) ds, \label{W-eql}
\end{align}
where we denote $\bV=\bR_x$ and $\bW=\bS_x$ and $(V_0,W_0) =(R'_0(\cdot ), S'_0(\cdot )) $ and  $N_{ju}, N_{jS}, N_{jR}$ ($j=1,2$) are partial derivatives of $N_j=N_j (u,R,S)$ with $u, S, R$ respectively (the same manners  are also used for $L$). 
From Lemma \ref{x-d-lm}, we obtain  that $| \partial_x x_{+} (0;t,x)|$ is bounded by $2$, if $T$ is small (note that smallness of $T$ depends on $A_0, A_1$ and  $A_2$). Hence we have that
\begin{align}
|W(t,x)| \leq & \frac{A_4\braket{x}^{-\gamma}}{2} + 2\int_0 ^t \left|N_{1u} v_x + N_{1R} \bV + N_{1S}  \bW\right|  (t,x_- (s;t,x)) ds  \notag \\
&+ 2\int_0 ^t  |L_u v_x  + L_R \bV + L_S \bW | (t,x_- (s;t,x)) ds \label{es-vw}.
\end{align}
From \eqref{0con2}-\eqref{0con4},  $| N_{1R} \bV|$ and $|N_{1S}  \bW|$ are trivially estimated as
\begin{align*}
| N_{1R} \bV|+ |N_{1S}  \bW| \leq & \frac{C}{v}(|\bR| + |\bS|)(|\bV| + |\bW|) \\
\leq & C\braket{x}^{\al -\bt -\gamma}\\
\leq & C \braket{x}^{ -\gamma}.
\end{align*}
From \eqref{0con2}-\eqref{0con4} and \eqref{ab1} and \eqref{ab2}, we have that for $|N_{1u} v_x|$
\begin{align*}
|N_{1u} v_x| \leq & \frac{Cv^a v_x}{u^{2+a}}(|\bR|^2 + |\bS|^2) \\
\leq & \frac{C}{u^{2+a}}(|\bR|^3 + |\bS|^3) \\
\leq & C\braket{x}^{(2+a)\al -3\bt}\\
\leq & C\braket{x}^{-\gamma}.
\end{align*}
From  \eqref{ab1}, \eqref{ab2}, \eqref{f-con} and \eqref{f-con2}, we estimate $L_u v_x$ as
\begin{align*}
|L_u v_x| \leq & C\left(\frac{|F(v) v_x|}{v^{a+1}} + \frac{|F'(v) v_x|}{v^{a}} \right) \\ 
\leq & \frac{C(|R|+|S|)}{v^a} \\
\leq & C\braket{x}^{\al a - \bt}\\
\leq & C \braket{x}^{-\gamma}.
\end{align*}
\eqref{f-con} and \eqref{0con4} directly implies that
\begin{align*}
 |L_R \bV| + |L_S \bW | \leq C\braket{x}^{-\gamma}.
\end{align*}
Applying these estimates to \eqref{es-vw},  we obtain that with small $T$
\begin{align}
\|\braket{x}^{\gamma} W  \|_{L^\infty _T L^\infty} \leq   \frac{A_4}{2} . \label{w-es}
\end{align}
Similarly 
\begin{align}
\|\braket{x}^{\gamma} V  \|_{L^\infty _T L^\infty}  \leq \frac{A_4}{2} . \label{v-es}
\end{align}
Estimates of $R_t$ and  $S_t$ are obtained from \eqref{lfs}. In fact, we have that
\begin{align}
|R_t (t,x)| + |S_t (t,x)|\leq & |v^a R_x| + |v^a S_x|+ |N_1(v \bR, \bS)| \notag \\
&+ |N_2(v \bR, \bS)| + 2|L(v \bR, \bS)|  \notag\\
\leq & C\braket{x}^{-\gamma} + C\braket{x}^{-\beta}\notag  \\
\leq & C_A\braket{x}^{-\gamma}, \label{rst-es}
\end{align} 
where $C_A$ is a positive constant depending on $A_1$, $A_3$ and $A_4$ (independent of $A_5$).
Here we choose $A_5$ as $A_5 =C_A$. From  the above estimates, we have that $(R,S) \in Y_{\bt, 1} \times Y_{\bt, 2}$.
Next we show that $u \in X_\al.$
From \eqref{def-u0}, \eqref{R-es} and \eqref{S-es}, it follows for sufficiently small $T$ that
\begin{align}
\braket{x}^{\al} u (t) \geq & \braket{x}^{\al} u_0 (x) - \int_0 ^t\dfrac{ \braket{x}^{\al}(|R|+|S|)}{2} ds  \notag \\
\geq & 2A_0 - \int_0 ^t  \dfrac{ \braket{x}^{\bt}(|R|+|S|)}{2} ds  \notag \\
\geq & 2A_0-   T A_3 \notag  \\
\geq  & A_0 . \label{u-es}
\end{align}
Similarly we can easily check that $\|  u \|_{L^\infty _T L^\infty} \leq A_1 $, if $T$ is small. 
Next we show that
\begin{align} \label{a2-con}
\| \braket{x}^\beta u^a u_x \|_{L^\infty _T L^\infty} \leq A_2,
\end{align}
\eqref{def-u0} directly implies that
\begin{align}
u^a  u_x =& \left(u_0 + \int_0 ^t \frac{R+S}{2} ds \right)^a \notag  u'_0  \\
+ & \left(u_0 + \int_0 ^t \frac{R+S}{2} ds \right)^a \int_0 ^t \frac{R_x + S_x}{2} ds .\label{u-ese}
\end{align}
From \eqref{0con2} and the boundedness of $\braket{x}^{\bt}u^a _0 |u' _0|$  the first term of \eqref{u-ese} is estimated as
\begin{align*}
\left(u_0 + \int_0 ^t \frac{R+S}{2} ds \right)^a  |u'_0| \leq  & 2^a (u^a _0   + C T^a \braket{x}^{-a\bt}) |u'_0|  \\
\leq & (2^2 + CT^a) u^a _0 |u' _0| \\
\leq & C_{1,A} \braket{x}^{-\bt},
\end{align*}
where  we note that the positive constant $C_{1,A}$ does not depend on $A_2$.
Deducing  $R_x +S_x $ via \eqref{lfs}, we also obtain
\begin{align*}
\left(u_0 + \int_0 ^t \frac{R+S}{2} ds \right)^a \left|\int_0 ^t\frac{R_x + S_x}{2} ds \right| \\
=\left(u_0 + \int_0 ^t \frac{R+S}{2} ds \right)^a  \left| \int_0 ^t\frac{1}{2v^a}(R_t - S_t + N_2 - N_1)  ds \right| \\
\leq 2^a (u^a _0 + T^a \braket{x}^{-a\bt}) \left| \int_0 ^t\frac{1}{2v^a} (R_t - S_t + N_2 - N_1  )ds\right|  .
\end{align*}
Since the spatial decay of $R_t$ and $S_t$ is not enough to show \eqref{a2-con}, we need to change this term.
From the integration by parts and the property of $X_\al$ that $v_0 =u_0$, we obtain that
\begin{align}
\int_0 ^t\frac{R_t - S_t}{2v^a} ds=\frac{R-S}{2v^a} - \frac{R_0-S_0}{2u^a _0} + \int_0 ^t \frac{a(R-S)v_t}{2v^{a+1}} ds. \label{rs-hen}
\end{align}
While, from \eqref{c1-3}, we have that
\begin{eqnarray*}
 u_0 -  A_3 T \braket{x}^{-\bt} \leq v\leq u_0 + A_3 T \braket{x}^{-\bt}.
\end{eqnarray*}
Hence, with the help of \eqref{c1}, it holds that
\begin{align} \label{vv0}
\frac{1}{2} \leq 1 - A_3 A_0T \leq \frac{v}{u_0} \leq 1 + A_3 A_0T \leq 2
\end{align}
for small $T$. The third term of the right hand side in \eqref{rs-hen} is estimated as 
\begin{align}
\int_0 ^t \frac{a(R-S)|v_t|}{2v^{a+1}} ds \leq  &C \int_0 ^t \frac{ \braket{x}^{\al}(R-S)|v_t|}{v^{a}} ds \notag \\
\leq C & \braket{x}^{-\bt} \int_0 ^t  v^{-a} ds. \label{rs-hen2}
\end{align}
Applying \eqref{vv0} and \eqref{rs-hen2} to \eqref{rs-hen}, we  have 
\begin{eqnarray*}
 2^a (u^a _0 + T^a \braket{x}^{-a\bt}) \left| \int_0 ^t\frac{1}{2v^a}( R_t - S_t ) ds\right| \leq C_{2,A}  \braket{x}^{-\bt}
\end{eqnarray*}
and similarly 
\begin{eqnarray*}
2^a (u^a _0 + T^a \braket{x}^{-a\bt}) \left| \int_0 ^t\frac{1}{2v^a}( N_2 - N_1 )   ds\right| \leq C_{3,A}  \braket{x}^{-\bt},
\end{eqnarray*}
where  positive constants $C_{2,A}$ and $C_{3,A}$ are independent of $A_2$.
\begin{eqnarray*}
\| \braket{x}^{-\bt} u^a u_x \|_{L^\infty _T L^\infty} \leq C_{1,A}+ C_{2,A} + C_{3,A}.
\end{eqnarray*}
Taking $A_5 = C_{1,A}+ C_{2,A} + C_{3,A}$, we obtain \eqref{a2-con}.  \eqref{def-u0} and \eqref{c2} directly yield that
\begin{align*}
\| \braket{x}^{-\bt}  u_t \|_{L^\infty _T L^\infty}  \leq & \left\| \frac{\braket{x}^{-\bt} (R + S)}{2}  \right\|_{L^\infty _T L^\infty}  \\
\leq & A_3.
\end{align*}
Therefore we have that $(u,R,S) \in X_\al  \times Y_{\beta, 1} \times Y_{\beta, 2}$.
In the end of the proof,  we show that $(u,R,S)$ is Lipschitz continuous. From  \eqref{w-es}, \eqref{v-es} and \eqref{rst-es}, we can obviously check that  $R$ and $S$ satisfies the following uniform Lipschitz estimate:
\begin{align} \label{rs-lip}
|R(t,x)-R(s,y)|+|S(t,x)-S(s,y)|\leq 2(A_4 + A_5)(|t-s|+|x-y|).
\end{align}
Next we check that $u$ is  Lipschitz continuous.
From \eqref{a2-con} and the condition that $\beta \geq a \alpha$, we have that
\begin{align*}
|u(t,x) - u (t,y)|\leq & \left|\int^x _y |u_x (t,z)| dz \right|\\
\leq & C \left| \int^x _y \braket{z}^{a\al} u^a |u_x| (t,z)dz \right| \\
\leq & C|x-y|.
\end{align*} 
Combining this estimate with  the boundedness of $R$ and $S$, we have from \eqref{def-u0}  that
$u$ is  Lipschitz continuous  such that
for any $t_1 ,t_2 \in [0,T]$, $x_1, x_2 \in \R$ with
\begin{align} \label{m1}
|u(t_1,x_1) - u(t_2,x_2)| \leq C(|t_1 - t_2| + |x_1 - x_2|),
\end{align}
where C is a positive constant depending on $A_0, A_3, A_4$ and  $A_5$. These additional properties are used to show that the fixed point of $\Phi$ satisfies integral equations.
\end{proof}
 \begin{Prop} \label{pr-con}
Under the same assumptions on \eqref{phi-def}, $\Phi$ is a contraction mapping in the topology of $L^\infty$-norm with  small $T>0$. Namely, If $T>0$ is small, then there exists a constant $c \in (0,1)$ such that  $\Phi$ satisfies that 
\begin{eqnarray*}
\|u_1 - u_2  \|_{L^\infty _T L^\infty} + \| R_1 - R_2  \|_{L^\infty _T L^\infty} +\| S_1 - S_2  \|_{L^\infty _T L^\infty} \\
\leq c \left(  \|v_1 - v_2  \|_{L^\infty _T L^\infty} + \| \bR_1 - \bR_2  \|_{L^\infty _T L^\infty} +\| \bS_1 - \bS_2  \|_{L^\infty _T L^\infty}    \right),
\end{eqnarray*}
where  $(u_j ,R_j , S_j) = \Phi (v_j, \bar{R}_j, \bar{S}_j)$ with $j=1, 2$.
\end{Prop}
\begin{proof}
Put  $\tilde{u}=u_1 - u_2 $,  $\tilde{R}=R_1 -R_2$, $\tilde{S}=S_1 -S_2$.
From \eqref{lfs}, we have
\begin{align*}
\tilde{R}_t - v^a _1 \tilde{R}_x  =&N_1(v_1, \bR_1, \bS_1) - N_1(v_2, \bR_2, \bS_2) \\
&+L(v_1, \bR_1, \bS_1) - L(v_2, \bR_2, \bS_2) \\
&+ (v^a _1 - v^a _2)R_{2x} .
\end{align*}
From the method of characteristic, we have that
\begin{align}
\tilde{R}(t,x)=&\int_0 ^t  (N_1(v_1, \bR_1, \bS_1) - N_1(v_2, \bR_2, \bS_2))  ds \notag \\
 &+ \int_0 ^t (L(v_1, \bR_1, \bS_1) - L(v_2, \bR_2, \bS_2)) ds \notag \\
&+  \int_0 ^t (v^a _1 - v^a _2)R_{2x} ds. \label{r-t}
\end{align}
The second term of the right hand side in \eqref{r-t} can be written as
\begin{align}
 \int_0 ^t (L(v_1, \bR_1, \bS_1) - L(v_2, \bR_2, \bS_2)) ds  =& \int_0 ^t (G(v_1)-G(v_2))(\bR_1  - \bS_1)ds  \notag \\
 & + \int_0 ^t G(v_2) (\bR_1 - \bS_1 - \bR_2 + \bS_2) ds ,\label{L-t}
\end{align}
where we set $G(\theta)=F(\theta)/2\theta^a$.
Using \eqref{f-con},  \eqref{f-con2} and \eqref{c1} , we obtain that 
\begin{eqnarray*}
|G(v_1) - G(v_2)| \leq \int_0 ^1 |G'(\theta v_1 + (1-\theta) v_2) |d\theta |v_1 - v_2| \\
\leq \int_0 ^1 \frac{C d\theta}{\theta v_1 +(1-\theta) v_2} |v_1 - v_2|  \\
\leq C  \braket{x}^{\al} |v_1 - v_2|,
\end{eqnarray*}
which implies that the first term of the right hand side in \eqref{L-t} is estimated as
\begin{eqnarray*}
\int_0 ^t |G(v_1)-G(v_2)||\bR_1  - \bS_1|ds \leq   CT A_3 \|v_1 - v_2\|_{L^\infty _T L^\infty } .
\end{eqnarray*}
From \eqref{f-con}, the second term is estimated by
\begin{eqnarray*}
\int_0 ^t |G(v_2)| |\bR_1 - \bS_1 - \bR_2 + \bS_2| ds \leq CT (\|\bR_1 - \bR_2\|_{L^\infty _T L^\infty } + \|\bS_1 - \bS_2\|_{L^\infty _T L^\infty }).
\end{eqnarray*}
Setting $N_1(v,\bR, \bS) =\frac{a}{2v}Q(\bR, \bS)$ and $Q(\bR, \bS)= (\bR^2 -\bR\bS)$,
 we change the first term of the right hand side  in \eqref{L-t} to
\begin{eqnarray}
\int_0 ^t  (N_1(v_1, \bR_1, \bS_1) - N_1(v_2, \bR_2, \bS_2))  ds=\int_0 ^t a\left(\dfrac{v_2 -v_1}{2v_1 v_2}  \right) Q(\bR_1, \bS_1)  ds \notag \\
 + \int_0 ^t \dfrac{a}{2 v_2}   \left( Q(\bR_1, \bS_1) -  Q(\bR_2, \bS_2) \right) ds. \label{r-t2}
\end{eqnarray}
The first term of the right hand side in \eqref{r-t2} is estimated as
\begin{eqnarray*}
\left|  \int_0 ^t a\left(\dfrac{v_2 -v_1}{2v_1 v_2}  \right) Q(\bR_1, \bS_1)  ds \right| \leq \int_0 ^t a\dfrac{|v_1 - v_2|}{2\braket{x}^{2\al} |v_1 v_2|} \braket{x}^{2\al}|Q(\bR_1, \bS_1)| ds \\
\leq \int_0 ^t a\dfrac{2|v_1 - v_2|}{A^2 _0} \braket{x}^{2\bt}|Q(\bR_1, \bS_1)| ds \leq CT \| v_1 - v_2 \|_{L^\infty _T L^\infty} .
\end{eqnarray*}
The second term can be estimated as
\begin{eqnarray*}
\left| \int_0 ^t \left(\dfrac{a}{2 v_2}  \right) \left( Q(\bR_1, \bS_1) -  Q(\bR_2, \bS_2) \right) ds \right| \leq \int_0 ^t \dfrac{a}{2A_0} \braket{x}^{\al} \left|Q(\bR_1, \bS_1) -  Q(\bR_2, \bS_2)\right| ds \\
\leq \int_0 ^t \dfrac{a}{A_0} \braket{x}^{\bt} \left|Q(\bR_1, \bS_1) -  Q(\bR_2, \bS_2)\right| ds \\
\leq CT \left( \| \bar{R}_1 - \bar{R}_2 \|_{L^\infty _T L^\infty} + \| \bar{S}_1 - \bar{S}_2 \|_{L^\infty _T L^\infty}\right) .
\end{eqnarray*}
Next we estimate the third term of the right hand side in \eqref{r-t2}. 
When $a \geq1$,  from the mean-value theorem for $|v^a _1 - v^a _2|$ and the boundedness of $R_x$, we obtain that
\begin{eqnarray*}
|(v^a _1 - v^a _2) R_{2x}| \leq C\|v_1 - v_2\|_{L^\infty _T L^\infty}.
\end{eqnarray*} 
While, $a<1$, by using the boundedness of $\braket{x}^{\gamma}R_x$, we have that
 \begin{align*}
|(v^a _1 - v^a _2) R_{2x}| \leq & a \int_0 ^1  (\theta v_1 +(1-\theta)v_2)^{a-1} |v_1 - v_2| |R_{2x}|d\theta \\
& \leq C \braket{x}^{\gamma}  |v_1 - v_2| |R_{2x}| \\
& \leq C\|v_1 - v_2\|_{L^\infty _T L^\infty}.
\end{align*} 
Therefore, we obtain that for sufficiently small $T$
\begin{eqnarray*}
\| \tilde{R} \|_{L^\infty _T L^\infty} \leq \dfrac{1}{6} \left( \| v_1 - v_2 \|_{L^\infty _T L^\infty} + \| \bar{R}_1 - \bar{R}_2 \|_{L^\infty _T L^\infty} + \| \bar{S}_1 - \bar{S}_2 \|_{L^\infty _T L^\infty}\right).
\end{eqnarray*}
In the same way as in the estimate of $\tilde{R}$, we have that
\begin{eqnarray*}
\| \tilde{S} \|_{L^\infty _T L^\infty} \leq \dfrac{1}{6} \left( \| v_1 - v_2 \|_{L^\infty _T L^\infty} + \| \bar{R}_1 - \bar{R}_2 \|_{L^\infty _T L^\infty} + \| \bar{S}_1 - \bar{S}_2 \|_{L^\infty _T L^\infty}\right).
\end{eqnarray*}
From \eqref{def-u0}, the above two estimates on $\tilde{R}$ and $\tilde{S}$ imply that for sufficiently small $T$
\begin{align*} 
\| \tilde{u} \|_{L^\infty _T L^\infty} \leq & \dfrac{T}{2} \left( \| \tilde{R} \|_{L^\infty _T L^\infty} +  \| \tilde{S} \|_{L^\infty _T L^\infty}  \right) \\
\leq & \dfrac{1}{6} \left( \| v_1 - v_2 \|_{L^\infty _T L^\infty} + \| \bar{R}_1 - \bar{R}_2 \|_{L^\infty _T L^\infty} + \| \bar{S}_1 - \bar{S}_2 \|_{L^\infty _T L^\infty}\right).
\end{align*}
Therefore, we find that $\Phi$ is a contraction mapping for sufficiently small $T>0$.
\end{proof}
Next we construct a unique solution $(u,R,S)$ of the nonlinear problem  and the characteristic curves $x_{\pm} (\cdot ;t,x)$.
\begin{Prop} \label{pr-non}
Under the same assumptions as in Proposition \ref{pr-con}, if $T$ is small, then there uniquely exist $(u,R,S) \in X_{\al} \times Y_{\beta, 1} \times Y_{\beta, 2} $ and $x_{\pm}(s) =x_{\pm} (s ;t,x)$ 
satisfying that
\begin{eqnarray}\label{nfs2}
\left\{  
\begin{array}{ll} 
R(t,x) = R(0,x_- (0)) + \int_0 ^t N_1(u, R, S)(s,x_- (s))+ L(u,R,S)(s,x_- (s)) ds,\\
S(t,x) = S(0,x_+ (0)) + \int_0 ^t N_2(v,R,S)(s,x_+ (s)) + L(u,R,S)(s,x_+ (s))ds,
\end{array}
 \right.
\end{eqnarray} 
\begin{eqnarray} \label{def-u-new}
u(t,x)=u_0 (x) + \int_0 ^t \frac{R+S}{2} ds 
\end{eqnarray}
and
\begin{eqnarray} \label{int-xu}
x_{\pm} (s; t,x) = x \pm \int_t ^s  u^a (\tau, x_{\pm} (\tau; t,x)) d\tau. 
\end{eqnarray}
\end{Prop}
\begin{proof}
We fix $K \geq 1$ arbitrarily.
From Proposition \ref{phi-def}, we can define a  sequence $\{  u_n, R_n, S_n\}_{n \in \N}$ in  $X_\al \times Y_{\bt,1} \times Y_{\bt,2}$ such that 
\begin{eqnarray*}
(u_{n+1}, R_{n+1}, S_{n+1})= \Phi(u_n, R_n, S_n)
\end{eqnarray*}
with initial term $(u_0, S_0, R_0)$.
 By Proposition \ref{pr-con}, $(u_n, R_n, S_n)$ converges the fixed point $(u,R,S)$ in the topology of $L^\infty$.
While we can define a sequence of the characteristic curves $\{ x_{\pm, n} (\cdot ;t,x) \}_{n \in \N}$. We note that the characteristic curves can be defined uniquely on $[0,T]$ with arbitrarily fixed $(t,x)$ by the  Lipschitz continuity  and the boundedness of $u_n ^a$. For arbitrarily fixed $K \geq 1$,  we see that  $\{ x_{\pm, n} (\cdot ;t,x) \}_{n \in \N}$ is the uniform equicontinuous  and uniform bounded from Lemma \ref{lip-x-lm} and \eqref{x-es}.
Thus  the Arzel\'a-Ascoli theorem implies that there exists a subsequence of  $\{ x_{\pm, n} (\cdot ;t,x) \}_{n \in \N}$ (we use the same suffix as in the original sequence) such that
$x_{\pm, n} (\cdot )$ converges $x_\pm (\cdot )$ uniformly on $[0,T] \times [0,T] \times [-K,K]$ as $n \rightarrow \infty$. 
Note that this choice of the subsequence is depending on $K$. However,
from Cantor's diagonal argument, we can reselect a subsequence independently of $K$ such that the convergence holds on $[0,T] \times [0,T] \times [-K',K']$
with any $K' \geq 1$.
From \eqref{rs-lip} and \eqref{m1}, we see that as $n\rightarrow \infty$
\begin{eqnarray*}
(u_n (t,x_{\pm,n}(t)),R_n (t,x_{\pm,n}(t)), S_n (t,x_{\pm,n}(t)))\rightarrow (u (t,x_{\pm}(t)),R (t,x_{\pm}(t)), S (t,x_{\pm}(t))).
\end{eqnarray*}
Hence \eqref{nfs2} and \eqref{int-xu} are satisfied.
Now we check that $(u,R,S) \in X_\al \times Y_{\bt,1} \times Y_{\bt,2}.$
It is obvious that the properties  \eqref{c1} and \eqref{c2} are satisfied.
From the Lipschitz continuity, $u,R,S$ are differentiable almost everywhere.
In the same way as in the proof of Proposition \ref{phi-def}, we can obtain the boundedness of 
$\braket{x}^{\bt} u^a u_x$ and $\braket{x}^{\bt}  u_t$, since the constant $A_2$ is taken independently of $A_4$ and $A_5$.
Thus we have $u \in X_\al$.
To show the boundedness of $\braket{x}^{\gamma}  R_x$  and $\braket{x}^{\gamma}  S_x$, differentiating the both side of \eqref{nfs2} with $x$, we obtain  that
\begin{align}
V (t,x) & = V_0 (x_- (0;t,x)) \partial_x x_{-} (0;t,x)  \notag \\
 + &\int_0 ^t \partial_x x_{-} (s;t,x) \left(N_{1u} u_x + N_{1R} V + N_{1S}  W\right) (t,x_- (s;t,x)) ds \notag \\
 +&  \int_0 ^t \partial_x x_{-} (s;t,x) \left(L_u u_x  + L_R V + L_S W \right) (t,x_- (s;t,x)) ds \label{V-eq}
\end{align}
and
\begin{align}
W (t,x) & = W_0 (x_- (0;t,x)) \partial_x x_{+} (0;t,x)  \notag \\
 + &\int_0 ^t \partial_x x_{+} (s;t,x) \left(N_{2u} u_x + N_{2R} W + N_{2S}  V\right) (t,x_+ (s;t,x)) ds \notag \\
 +&  \int_0 ^t \partial_x x_{+} (s;t,x) \left(L_u u_x  + L_R V + L_S W \right) (t,x_+ (s;t,x)) ds. \label{W-eq}
\end{align}
In the same way as in the proof of \eqref{w-es} and \eqref{v-es}, we achieve the boundedness of $\braket{x}^{\gamma} W$
and $\braket{x}^{\gamma} V$. The estimates of $\braket{x}^{\gamma} R_t$ and $\braket{x}^{\gamma} S_t$ can be shown by similarity
to in \eqref{rst-es}. Thus we have  $(R,S) \in Y_{\bt,1}\times Y_{\bt,2}.$ The uniqueness can be shown in the same way as in the proof of
Proposition \ref{pr-con}.
\end{proof}
In the discussions so far, we do not assume any relations between $u_0$ and $R_0, S_0$. To show $u$ in \eqref{def-u-new} is a solution of \eqref{req}, we assume that
\begin{eqnarray}\label{fs20}
 \left\{  
\begin{array}{ll}
u' _0  =\frac{R_0 - S_0}{2 u^a _0},\\
u_1 =\frac{R_0 +S_0}{2}.
\end{array}
\right.
\end{eqnarray} 
Moreover, we  improve the regularity of the solution if $(u_0, u_1) \in C^2 \times C^1 $.
The following proposition completes the proof of Theorem \ref{main}.
\begin{Prop} \label{last-pr}
Addition to the assumption of Proposition \ref{pr-non}, we assume for $(u_0, u_1) \in C^2 \times C^1 _b$ that \eqref{fs20} is satisfied.
Then the function $u$ defined in \ref{def-u-new} is $C^2$ on $[0,T]\times \R$ and is the classical solution of \eqref{req}. 
\end{Prop}
\begin{proof}
From the Lipschitz continuity of $R,S$, these are differentiable  almost everywhere and satisfy that
\begin{eqnarray}\label{fs2}
 \left\{  
\begin{array}{ll}
R_t -u^a R_x=N_1(u,R,S) + L(u,R,S) , \\
S_t +u^a S_x =N_2 (u,R,S) + L(u,R,S).
\end{array}
\right.
\end{eqnarray} 
Since $u$ is also differentiable almost everywhere, differentiating  \eqref{def-u-new}, we have that
\begin{eqnarray} \label{u-bibun}
\partial_x u = u' _{0} (x) + \int_0 ^t \frac{R_x+S_x}{2} ds.
\end{eqnarray}
From the first and third equation of \eqref{fs2}, we have that
\begin{align}
\int_0 ^t R_x +S_x ds = &  \int_0 ^t \frac{1}{u^a} \left( N_2(u,R,S)- N_1(u,R,S) + R_t - S_t \right) ds \notag \\
= &  \int_0 ^t \frac{1}{u^a} \left(  \frac{a}{2u}(S^2 - R^2) + R_t - S_t \right) ds .\label{henkei1}
\end{align}
From the integration by parts, $u_t = \frac{R+S}{2}$ and \eqref{fs20}, it follows that
\begin{align}
\int_0 ^t \frac{1}{u^a} \left(  R_t - S_t \right) ds = & -2u'_0 (x) + \frac{R-S}{u^a}+\int_0 ^t  \frac{a u_t}{u^{a+1}}  \left(  R - S \right) ds \notag \\
=&  -2u'_0 (x) +\frac{R-S}{u^a} \notag \\
+ & \int_0 ^t  \frac{a}{2u^{a+1}}  \left(  R^2 - S^2 \right) ds. \label{r-sx}
\end{align}
From \eqref{u-bibun},  \eqref{henkei1}  and \eqref{r-sx}, we have that
\begin{eqnarray} \label{ux-rel}
\partial_x u = \frac{R-S}{2u^a}.
\end{eqnarray}
Combining \eqref{fs2}, \eqref{def-u-new} and \eqref{ux-rel}, we have that the function $u$ satisfies \eqref{req}.
Lastly, applying Theorem 4 in Douglis \cite{D}, we obtain the continuity of the $W=R_x$ and $V=S_x$.
From the equations of $R,S$, we see that $R_t$ and $S_t$ are also continuous. Hence we have the continuity of $u_{xx}, u_{tx}, u_{tt}$.
Therefore we have that $u \in C^2 ([0,T]\times \R)$.
\end{proof}

\section{Concluding remarks}

\subsection{Physical background}
We set a function $G$ as a primitive function of $F$ such that $G(0)=0$.
Integrating with $x$ over $[-\infty, x]$, we formally obtain that
\begin{eqnarray*}
\int_{-\infty} ^x u_{tt} dx = \partial_x \left(\frac{u^{a+1}}{a+1} \right) + G(u).
\end{eqnarray*}
Setting  $v=\int_{-\infty} ^x u_t dx$ and $\sigma (u) =u^{a+1} /a+1$, we have the following 1st order hyperbolic equation:
\begin{eqnarray} 
\left\{  \begin{array}{ll} \label{req-p}
u_t - v_x =0,\\
v_t - \partial_x (\sigma (u)) = G(u).
\end{array} \right.
\end{eqnarray} 
This equations govern the motion for one dimensional elastic waves with the case that the density of  material is equal to $1$. 
Unknown functions $u$ and $v$ describe the differentiations of the displacement $X$ with $x$ and $t$ respectively. Namely $u=X_x (t,x)$ and $v=X_t (t,x)$.
The first equation means the relation $u_t =X_{xt} =X_{tx}=v_x$. The second equation is Newton's second since $v_t$ is the acceleration.
From the definition of $u$,  $u$ is the strain (more precisely, $(1,1)$ component of the stain matrix) and $\sigma (u)$ is so-called stress-strain relation.
$G$ is a external force term depending only on the strain.
The detailed derivation with $G\equiv 0$ is given in Cristescu's book \cite{NC}.

\subsection{On the generalization of the main theorem}
The our existence theorem is also applicable to $ \partial^2 _t u = (c(u)^2 u_{x})_x + F(u)u_x$ under the following assumptions on $c(\cdot ) \in C ([0,\infty)) \cap C^2 ((0,\infty))$ and $F \in C ([0,\infty)) \cap  C^1 ((0,\infty)) $
\begin{eqnarray}
C_{1. K} \theta^a \leq c(\theta) \leq C_{2, K},   \label{c-con1} \\
| c'(\theta)| \leq  C_{3,K} \theta^{a-1} \label{c-con2}, \\
|c''(\theta)| \leq C_{4,K} \theta^{a-2}, \label{c-con3}
\end{eqnarray}
and
\begin{eqnarray} \label{f-con23}
 |F(\theta)| \leq C_{5, K} \theta^a, \\
 |F'(\theta)| \leq C_{6, K} \theta^{a-1}, \label{f-con24}
\end{eqnarray}
where $a \geq 0$, $\theta \in [0,K]$ for $K>0$ and $C_{j,K}$  are positive constants depending on $K$ for $j=1,\ldots, 6$. For this equation, the unknown valuable $R$ and $S$ are defined by
\begin{eqnarray*}
R= u_t + c(u) u_x, \\
S= u_t -c(u) u_x
\end{eqnarray*}
and $R$ and $S$ satisfy that
\begin{eqnarray}\label{fs-c}
 \left\{  
\begin{array}{ll}
\partial_t R -u^a\partial_x R=\frac{c'}{2c}(RS-S^2) + F(u)\frac{R-S}{2c} , \\
\partial_t S +u^a\partial_x S =\frac{c'}{2c}(RS-R^2) + F(u)\frac{R-S}{2c}.
\end{array}
\right.
\end{eqnarray} 
Since we have  that $\frac{|c'(u)|}{c(u)} \leq C \braket{x}^{\al} $ from the assumption on $c$ and initial data, we can obtain  weighted $L^\infty$ estimated for $R$ and $S$.
The assumption \eqref{c-con3} is used in the proof of the construction of the contraction mapping.

\subsection{Finite time blow-up or degeneracy}
We define $T^*$ as the maximal existence time of the solution constructed by Theorem \ref{main}.
When $T^* < \infty$, we have the following  criterion of the break-down:
\begin{eqnarray} \label{blow}
\limsup_{t\rightarrow T^*} \| \braket{x}^{\bt} u_t\|_{L^\infty } + \| \braket{x}^{\bt} u_x \|_{L^\infty } =\infty
\end{eqnarray}
or
\begin{eqnarray} \label{deg}
\liminf_{t\rightarrow T^*}  \inf_{\R} \braket{x}^{\al} u (t,x)=0.
\end{eqnarray}
We call \eqref{blow} and \eqref{deg} the blow-up and the degeneracy respectively. 
In the case that $F \equiv 0$, we can obtain the non-trivial solutions blow up in finite time, if $R(0,x)$ and $S(0,x)$ are non-negative.
In fact, we can show that the non-negativity of $R(0,x)$ and $S(0,x)$ preserves as time goes by, from which we have
$u_t (t,x) \geq 0$. Thus we find that \eqref{deg} does not occurs in finite time. Therefore, using the method of Lax \cite{lax} or \cite{z} (see also
Chen \cite{G}), we have the conclusion. While, in the case  that $F \equiv 0$, we can apply main theorems to the equation in \eqref{req} and 
find that \eqref{deg} can occur in finite time for non-trivial solution, if $R(0,x)$ and $S(0,x)$ are non-positive.
Sufficient conditions for the occurrence of \eqref{deg} have been studied in the author's papers \cite{s3, s4}.

\subsection{Multi-dimensional case}
The multi-dimensional version of the equation in \eqref{req} is
\begin{eqnarray*}
\partial^2 _t u = u^{2a} \Delta u + F(u) \cdot \nabla u =0.
\end{eqnarray*}
The method of characteristic (and Riemann invariant) does not work, even with radial initial data.
In the forthcoming paper, we deal this problem via a local-energy argument.

\end{document}